\documentclass{amsart}
\usepackage{amsmath,amsthm,amssymb,IMjournal}
\DeclareMathOperator*{\esssup}{ess\,sup}
\def\RR{\hbox{I\kern-.2em\hbox{R}}}

\begin{document}

\numberwithin{equation}{section}

\newtheorem{Th}{Theorem}[section]

\newtheorem{guess}[Th]{Theorem}
\newtheorem{uess}[Th]{Lemma}
\newtheorem{corollary}[Th]{Corollary}
\newtheorem{definition}[Th]{Definition}
\newtheorem{example}[Th]{Example}
\newtheorem{remark}[Th]{Remark}
\newtheorem{prop}[Th]{Proposition}


\title{On stability of linear neutral  differential equations with variable delays}

\author{\|Leonid  |Berezansky, |Beer-Sheva,
\|Elena |Braverman, |Calgary
}




\abstract 
We present a review of known stability tests and new explicit exponential stability conditions
for the linear scalar neutral equation with two delays
$$
\dot{x}(t)-a(t)\dot{x}(g(t))+b(t)x(h(t))=0,
$$
where
$$
|a(t)|<1,~  b(t)\geq 0, ~h(t)\leq t, ~g(t)\leq t,
$$
and for its generalizations, including equations with more than two delays,
integro-differential equations and equations with a distributed delay.
\endabstract



\thanks
E. Braverman was partially supported by the NSERC research grant RGPIN-2015-05976.   
\endthanks

\section{Introduction}

There are two different classes of neutral differential equations. 
The first one includes the  scalar linear equation
\begin{equation}\label{01}
(x(t)-a(t)x(g(t)))'=-b(t)x(h(t)),
\end{equation}
while the second class is represented by the equation
\begin{equation}
\label{1} 
\dot{x}(t)-a(t)\dot{x}(g(t))=-b(t)x(h(t)).
\end{equation}
The aim of the paper is to consider  explicit stability tests for equation (\ref{1})
and its generalizations, including integro-differential neutral equations
and neutral equations with a distributed delay. Equations (\ref{01}) and (\ref{1}) have
different sets of solutions, so a stability test for one of them  can be  applied 
to the other equation only if $a(t)\equiv a$, $g(t)=t-\sigma$, for some non-negative constants $a$ and $\sigma$. 
Here we focus on (\ref{1}) and only cite some interesting stability tests for equation (\ref{01})
to compare them with known results for equation (\ref{1}).

Stability theory for neutral equations of the second and higher order, systems and vector equations, stochastic neutral 
equations, nonlinear equations and mathematical models described by neutral equations are investigated in the monographs 
\cite{AzbSim, Burton, Gil, Gop, KolmMysh, KolmNos, Shaikhet}
and in numerous articles. 
However, such systems and equations are not in the framework of the present paper.
 
The following methods were used in stability investigations:
Lyapunov functions and functionals \cite{Gop, KolmMysh, KolmNos, Shaikhet},
fixed point methods \cite{Burton}, and application of the Bohl-Perron theorem \cite{AzbSim, Gil}. 
To obtain new stability tests, we apply the method based on the Bohl-Perron theorem together with 
a priori estimations of solutions, integral inequalities for fundamental functions of linear delay equations 
and various transformations of a given equation.

The paper is organized as follows. 
Section~2 contains a review of some known stability tests and methods applied to explore stability.
In Section 3 we present some  auxiliary statements which are later used to prove the main stability results for equation (\ref{1}) in Section~4.
Section~5 involves extensions of these results to some more general models, 
such as equations with several delayed terms of either neutral or non-neutral types, 
integro-differential equations and equations with a distributed delay. 
Section~6 presents a  discussion of the results, illustrating examples, 
as well as suggests some open problems and projects for future research.

\section{Review of Known Stability Tests}

In this section, we will give a review of most interesting results for
equation~(\ref{1}) and for its generalizations, which are the main objects of the present paper. However, we start with  stability results for  (\ref{01}) and its extensions, 
illustrating the state of the arts in the stability investigations for this class of equations and applicable for comparison with stability results for (\ref{1}).

\begin{prop}\label{proposition1} \cite{AgGr}
Consider the equation
\begin{equation}\label{02}
(x(t)+c(t)x(t-\tau))'+p(t)x(t)+q(t)x(t-\sigma)=0,
\end{equation}
where $\sigma\geq \tau$, $c,p,q \in C([t_0,\infty),[0,\infty))$, 
and, in addition, $c$ is differentiable with a locally bounded derivative.

Assume that there exist constants $p_1,p_2,q_1,q_2,c_1,c_2$ such that
$$
0\leq p_1\leq p(t)\leq p_2, ~~0\leq q_1\leq q(t)\leq q_2, ~~0\leq c(t)\leq c_1<1,~~ |c'(t)|\leq c_2.
$$
If at least one of the following conditions holds

a) $p_1+q_1>(p_2+q_2)(c_1+q_2\sigma)$;

b) $p_1>q_2+c_1(p_2+q_2)$
\\
then every solution of (\ref{02}) satisfies $\displaystyle \lim_{t \to +\infty} x(t) = 0$.
\end{prop}

 \begin{prop}\label{proposition2} \cite{Yu}
Consider the equation
\begin{equation}\label{03}
\left( x(t)-P(t)x(t-\tau) \right)'+Q(t)x(t-\sigma)=0,
\end{equation}
where $\tau, \sigma>0, P,Q\in C([t_0,\infty), \RR), Q(t)\geq 0$.

If 
$$
\int_{t_0}^{\infty} Q(s)ds=+\infty,~ |P(t)|\leq p< 1,~~ 
\limsup_{t\rightarrow\infty} \int_{t-\sigma} ^t Q(s)ds<\frac{3}{2}-2p(2-p)
$$
then equation (\ref{03}) is asymptotically stable.
\end{prop}

Proposition \ref{proposition2} is a nice result since for the non-neutral case $P(t)\equiv 0$ 
it is reduced to the best possible stability result with the constant $\frac{3}{2}$.

There are several improvements and extension of Proposition \ref{proposition2} 
\cite{TangZou, WangLiao, WuYu, Haiping}. In particular, the following result was obtained in \cite{TangZou}.
\begin{prop}\label{proposition2a} \cite{TangZou}
Assume that $\int_{t_0}^{\infty} Q(s)ds=+\infty,~ |P(t)|\leq p< 1$
and at least one of the following conditions hold:
\vspace{2mm}

a) $\displaystyle p<\frac{1}{4},~~\limsup_{t\rightarrow\infty} \int_{t-\sigma} ^t Q(s)ds<\frac{3}{2}-2p$; 
\vspace{2mm}

b) $\displaystyle \frac{1}{4}\leq p<\frac{1}{2},~~\limsup_{t\rightarrow\infty} \int_{t-\sigma} ^t Q(s)ds<\sqrt{2(1-2p)}$.
\vspace{2mm}

Then equation (\ref{03}) is asymptotically stable.
\end{prop}

Stability tests for equation (\ref{1}) will be classified according to the methods applied to obtain them.
The method of Lyapunov functions and functionals is the most popular tool in stability investigations for
all classes of functional differential equations and in particular for neutral equations.
The following results were obtained by this method.

\begin{prop}\label{proposition3} \cite[Theorem 5.1.1]{Gop}
Consider the equation
\begin{equation}\label{04}
\begin{array}{ll}
\dot{x}(t) & \displaystyle +\sum_{j=1}^m b_j\dot{x}(t-\sigma_j)+\beta\int_0^{\infty}K_2(s)\dot{x}(t-s)ds \\
& \displaystyle +a_0x(t)+\sum_{j=1}^n a_jx(t-\tau_j) +\alpha \int_0^{\infty}K_1(s)x(t-s)ds=0,
\end{array}
\end{equation}
where
$$
a_0>0,~ \tau_j\geq 0, \sigma_j\geq 0, a_j\tau_j\neq 0, ~~b_j\sigma_j\neq 0,~~ \int_0^{\infty}|K_i(s)|ds<\infty, ~~ 
\int_0^{\infty}s|K_i(s)|ds<\infty.
$$
Assume that
$$
\sum_{i=0}^n a_i+\alpha\int_0^{\infty}K_1(s)ds>0,
$$$$
\sum_{j=1}^m |b_j|+|\beta|\int_0^{\infty}|K_2(s)|ds+\sum_{i=1}^n |a_i|\tau_i+|\alpha|\int_0^{\infty}|K_1(s)|ds<1.
$$
Then all solutions of (\ref{04}) satisfy $\displaystyle \lim_{t\rightarrow\infty}x(t)=0$.
\end{prop}

Equation (\ref{04}) is autonomous. Many results for such models were obtained by analyzing
their characteristic equations (see \cite{KolmMysh,Kuang}, the recent paper \cite{Cahlon} and the bibliography therein). 

\begin{prop}\label{proposition4} \cite[Theorem 5.1.2]{Gop}
Consider the  non-autonomous equation
\begin{equation}\label{05}
\dot{x}(t)+a(t)x(t-\tau)+b(t)\dot{x}(t-\sigma)=0,
\end{equation}
where $a$ and $b$ are continuous functions. 

Assume that $\displaystyle \liminf_{t\rightarrow\infty}a(t)>0$,
$$
\begin{array}{ll}
\displaystyle \limsup_{t\rightarrow\infty} & \displaystyle \left[ \int_{t-\tau}^t [a(s+\tau)+a(s+2\tau)]ds +\frac{|b(t)|}{a(t+\tau)} \right. \\
\\
& \left.
\displaystyle + \int_{t-\tau}^t|b(s+\tau)|ds+4|b(t+\sigma+\tau)|a(t+\tau)\right]<2,
\\
\\
\displaystyle \limsup_{t\rightarrow\infty} & \displaystyle \left[4|b(t)b(t+\sigma)|+\int_{t-\tau}^t a(s+\tau)ds\right]<1.
\end{array}
$$
Then all solutions of (\ref{05}) satisfy $\displaystyle \lim_{t\rightarrow\infty}x(t)=0$.
\end{prop}

The following simple and nice test is a corollary of a stability result which was obtained for systems of neutral equations.
\begin{prop}\label{proposition5} \cite{Gop2}
Consider the equation 
\begin{equation}\label{06}
\dot{x}(t)=-a(t)x(t)+b(t)x(t-\tau)+c(t)\dot{x}(t-\sigma),
\end{equation}
where $a,b$ are continuous functions, $c$ is continuously differentiable. 
If $a(t)\geq a_0>0$, $|c(t)|\leq c_0<1$ and $|b(t)|<a_0$ then equation (\ref{06}) is asymptotically stable.
\end{prop}

We suggest that Proposition~\ref{proposition5} remains true if $\tau$ and $\sigma $  are  variable delays  
such that $ \displaystyle \lim_{t\rightarrow\infty} (t-\tau(t))=\infty$,
$\displaystyle \lim_{t\rightarrow\infty} (t-\sigma(t))=\infty$
but this conjecture is still an open problem.

The fixed point method was introduced to investigate stability by
Burton and his collaborators (see \cite{Burton}) and then applied to many functional differential
equations, including neutral equations. The following proposition is a typical result obtained by this method.

\begin{prop}\label{proposition6} \cite{Raffoul}
Consider the equation
\begin{equation}\label{07}
x'(t)=-a(t)x(t)-b(t)x(t-\tau(t))+c(t)x'(t-\tau(t)),
\end{equation}
where $a, b$   are continuous functions, $c$ is differentiable, $\tau$ is twice differentiable,
${\tau}'(t)\neq 1$, and $\displaystyle \int_0^{\infty} a(u)ds = +\infty$.
If there exists $\alpha \in (0,1)$ such that
$$
\left|\frac{c(t)}{1-{\tau}'(t)}\right|+\int_0^t e^{-\int_s^t a(u)du}\left|b(s)
+\frac{[a(s)c(s)+c'(s)](1-{\tau}'(s))+c(s){\tau}''(s)}
{(1-{\tau}'(s))^2}\right|ds\leq \alpha 
$$
then every solution of equation (\ref{07}) with a small continuous initial function tends to zero as $t\rightarrow+\infty$.
\end{prop}

Some other stability results for neutral equations obtained by the fixed point method can be found in
\cite{Ard, JinLuo, LiuYan, Zhao}.

In the recent monograph \cite{Gil} Gil' proved the Bohl-Perron theorem for many classes of linear functional differential equations
and obtained stability results for linear and nonlinear vector equations. The developed method is very original
and applies some operator and matrix inequalities. Here we cite a result for scalar 
neutral equations.

\begin{prop}\label{proposition7} \cite[Chapter 8]{Gil}
Consider the equation 
\begin{equation}\label{08}
\dot y(t)-a\dot y(t-\sigma)+ by(t-\tau)=[Fy](t),
\end{equation}
where $a, b$ are positive, while $\sigma \geq 0$, $\tau \geq 0$, 
$F$ is a continuous causal mapping acting from $L^2(t_0,\infty)$  into  $L^2(t_0, \infty)$ such that
$\|Fu\|_{L^2(t_0,\infty)}\leq q \|u\|_{L^2(t_0,\infty)}$ for some $q>0$.

Let  the equation $\lambda=\lambda e^{\sigma \lambda}a+ e^{\tau \lambda}b$
have  a  positive  root and $b>q$. 
Then equation (\ref{08}) 
is $L^2$-absolutely stable, i.e. any solution of this equation belongs to $L^2(t_0,\infty)$.
\end{prop}

Note that Proposition \ref{proposition7} is concerned with an asymptotic property of solutions
which is different from asymptotic stability.

The method based on the Bohl-Perron theorem was introduced in \cite{Ber} and then applied
to delay differential and impulsive equations (see, for example, \cite{ABB, AzbSim, BB3,GD}).
The following stability tests for a neutral equation are cited from \cite{AzbSim}.

Consider the equation
\begin{equation}\label{09}
\dot{x}(t)-\sum_{i=1}^n q_i(t)\dot{x}(g_i(t))+\sum_{k=1}^m p_k(t)x(h_k(t))=0,
\end{equation}
where all the parameters of the equation are measurable functions,
$0 \leq t-h_k(t)\leq \tau$, $0 \leq t-g_i(t)\leq \delta$, $\displaystyle 0<p_0\leq \sum_{k=1}^m p_k(t)\,$.
  
\begin{prop}\label{proposition8} \cite[Test 2.6.2, p. 78]{AzbSim}
If $\displaystyle \sum_{k=1}^m p_k(t)\leq \frac{1}{\tau e}$ and
$$
0\leq \sum_{i=1}^n q_i(t)\leq \left(1+\sup_{t\geq t_0} \frac{\sum_{k=1}^m |p_k(t)|}{\sum_{k=1}^m p_k(t)}\right)^{-1}
$$
then equation (\ref{09}) is exponentially stable.
\end{prop}
Denote
$$
p(t)=\sum_{k=1}^m p_k(t),~~ q(t)=\sum_{i=1}^n q_i(t),~~ r(t)=p(t)(1-q(t))^{-1}, ~~r_j(t)=r(g_j(t))/r(t),
$$$$
\tau_{h_k}(t)=\int_{h_k(t)}^t r(s)ds,~~ \tau_{g_k}(t)=\int_{g_k(t)}^t r(s)ds,~~ \sigma (\omega)=\int_0^{\infty}x_{\omega}(s)ds,
$$
where $x_{\omega}$, $0\leq \omega<\frac{\pi}{2}$ is the fundamental solution of the delay equation
$$
\dot{x}(t)+x(t-\omega)=0.
$$
It is known that $\displaystyle \sigma (\omega)=1$ if $0\leq \omega\leq\frac{1}{e}$, and $\displaystyle \lim_{\omega \to {\frac{\pi}{2}-}} \sigma (\omega) = +\infty$. 

\begin{prop}\label{proposition9} \cite[Test 2.6.3, p. 81]{AzbSim}
Assume that $q(t)\neq 1$ almost everywhere, \\
$\displaystyle
\liminf_{t\rightarrow\infty}\frac{p(t)}{1-q(t)}>0$, $\displaystyle \sum_{j=1}^n \|q_j r_j\|<1$, 
and there exists $\omega\in [0,\frac{\pi}{2})$ such that 
\begin{eqnarray*}
&\displaystyle \sum_{k=1}^m \|p_k\|\left(1-\sum_{j=1}^n \|q_j r_j\|\right)^{-1}\left\|\frac{1}{r}\right\|
\left(\left\|\frac{1}{r}\right\|\sum_{k=1}^m \|p_k(\tau_{h_k}-\omega)\|+\sum_{j=1}^n \|q_j \tau_{g_j}\|\right.
\\ \\
& \displaystyle +\left.\sum_{j=1}^n \|q_j(1-r_j)\|\right)<(\sigma(\omega))^{-1}\left(1-\sum_{j=1}^n\|q_j r_j\|\right).
\end{eqnarray*}

Then equation (\ref{09}) is exponentially stable.
\end{prop}

Here $\|\cdot\|$ is the usual essential supremum norm in the space $L_{\infty}[t_0,\infty)$.

\begin{corollary}\label{corollary01}\cite[Corollary 2.6.2, p. 85]{AzbSim}
Consider the autonomous equation
\begin{equation}\label{010}
\dot{x}(t)-q\dot{x}(t-\delta)+px(t-\tau)=0,
\end{equation}
where $|q|<1$, $p>0$. Assume that there exists $\omega\in[0,\frac{\pi}{2})$
such that
\begin{equation}\label{011}
(1-q)|p\tau+q\omega-\omega|+p|q|\delta+|q|(1-|q|)<\frac{(1-|q|)^2}{\sigma(\omega)}.
\end{equation}
Then equation (\ref{010}) is exponentially stable.

In particular, assuming $\omega=\frac{1}{e}$ in (\ref{011}), we obtain the exponential stability condition
\begin{equation}\label{add_star}
\frac{1-q}{1-|q|}\left|p\tau-(1-q)\frac{1}{e}\right|<1-2|q|-\frac{p|q|\delta}{1-|q|}.
\end{equation}
\end{corollary}

Every method used to investigate stability has its advantages and limitations. Some results were 
obtained by deep analysis of concrete equations, like Propositions  \ref{proposition2} and \ref{proposition2a}.
Such results usually have conditions close to the best possible ones, however this method 
can be applied only to a restricted class of equations.

The method of Lyapunov functions and functionals can be applied to all known
classes of functional differential equations including systems and nonlinear equations.
However, it is usually difficult to apply this method to equations with time-dependent delays.

The fixed point method is also quite universal, but stability conditions obtained
by this method are sometimes rather restrictive (delay functions should be twice differentiable)
and far from the best known tests for partial classes of equations. 

The method proposed by Gil' for vector delay differential equations
gives new results even for ordinary differential equations by application of
some matrix functions like the logarithmic matrix norm ($\mu$-norm).

The method based on the Bohl-Perron theorem leads to new stability tests for all classes of linear functional differential equations. 
The advantage of this method
is that the stability problem is reduced to estimation of the norm or of the spectral radius
for some linear operators in functional spaces on the half-line. However, this method is
not applicable to nonlinear differential equations.
New results in the present paper are obtained using the Bohl-Perron theorem.

\section{Preliminaries}

We consider scalar delay differential equation (\ref{1})
under the following conditions:\\
(a1) $a, b, g, h$ are Lebesgue measurable  essentially
bounded functions on $[0,\infty)$;\\
(a2) $ \mbox{ess}\sup_{t\geq t_0}  |a(t)|\leq a_0<1$  for some $t_0\geq 0$, $b(t)\geq 0$;\\ 
(a3) $g(t)\leq t$, $\displaystyle \lim_{t\rightarrow\infty}g(t)=\infty$, $mes~ E=0\Longrightarrow mes~ g^{-1}(E)=0$,
where $mes~E$ is  the Lebesgue
measure of the set $E$;\\
(a4) $h(t)\leq t$, $\displaystyle \lim_{t\to\infty}h(t)=\infty.$

Together with  (\ref{1}) we consider for each $t_0 \geq 0$ an initial value problem
\begin{equation}
\label{2}
\dot{x}(t)-a(t)\dot{x}(g(t))+b(t)x(h(t))=f(t), ~~t\geq t_0,
\end{equation}

\begin{equation}
\label{3}
x(t)=\varphi(t), ~ t \leq t_0,~\dot{x}(t)=\psi(t),~ t<t_0,~ 
\end{equation}
where $f$, $\varphi$ and $\psi$ satisfy the assumption:
\\
(a5) $f:[t_0,\infty)\rightarrow {\mathbb R}$ is a Lebesgue measurable locally essentially
bounded  function, 
$\varphi,\psi :(-\infty,t_0)\rightarrow {\mathbb R}$ are Borel
measurable bounded functions.

In the main part of the paper we also assume that the delays are bounded:
\\
(a6) $t-g(t)\leq \delta$, $t-h(t) \leq \tau$ for $t \geq t_0$ and some $\delta>0$, $\tau>0$ and $t_0 \geq 0$.

\begin{definition} 
A function $x: {\mathbb R} \rightarrow {\mathbb R}$ is called {\bf a solution of problem}  (\ref{2}),(\ref{3}) 
if it is absolutely continuous on each interval $[t_0,c]$, satisfies equation (\ref{2}) for almost all $t\in [t_0,\infty)$ and equalities (\ref{3}) for $t\leq t_0$.
\end{definition}

There exists one and only one solution of problem (\ref{2}),(\ref{3}), see \cite{AzbSim}.

Consider the initial value problem
\begin{equation}
\label{4}
\dot{x}(t)+b(t)x(h(t))=f(t), ~x(t)=0,~t\leq t_0,
\end{equation}
where $b(t), f(t)$ and  $h(t) \leq t$ are Lebesgue measurable locally bounded functions.

\begin{definition} 
For each $s\geq t_0$ the solution $X(t,s)$ of the problem
\begin{equation}
\label{5}
\dot{x}(t)+b(t)x(h(t))=0, ~x(t)=0,~\dot{x}(t)=0,~t<s,~x(s)=1
\end{equation}
is called {\bf a fundamental function of equation}  (\ref{4}).
\end{definition} 
We assume $X(t,s)=0$ for  $0\leq t<s$.

The same definition will be used for other classes of linear functional differential equations, including neutral equations.

\begin{uess}  \cite{AzbSim}
\label{lemma2}
The solution of
problem  (\ref{4})  can be presented in the form
\begin{equation}
\label{6}
x(t)=\int_{t_0}^t X(t,s)f(s)ds.
\end{equation}
\end{uess}

\begin{definition} 
We will say that equation (\ref{1}) is {\bf (uniformly) exponentially stable} 
if there exist positive numbers $M$ and $\gamma$ such that 
the solution of problem  (\ref{2}),(\ref{3}) has the estimate 
\begin{equation}\label{7}
|x(t)|\leq M e^{-\gamma (t-t_0)} \sup_{t \in (-\infty,  t_0]}(|\varphi(t)|+|\psi(t)|),~~t\geq t_0,
\end{equation}
where $M$ and $\gamma$ do not depend on $t_0 \geq 0$, $\varphi$ and $\psi$.
The fundamental function $X(t,s)$ of equation (\ref{1}) {\bf has an exponential estimate} if it satisfies
$$
|X(t,s)|\leq M_0 e^{-\gamma_0(t-s)}, ~~ t\geq s\geq t_0
$$
for some positive numbers $M_0>0$ and $\gamma_0>0$. 
\end{definition}

Existence of an exponential estimate for the fundamental function is equivalent \cite{AzbSim} 
to the exponential stability for equations with bounded delays. 
The following result is usually referred to as the Bohl-Perron principle.

\begin{uess}\label{lemma3}\cite[Theorem 4.7.1]{AzbSim}
Assume that (a1)-(a4),(a6) hold and  the solution of the problem 
\begin{equation}\label{10}
\dot{x}(t)-a(t)\dot{x}(g(t))+b(t)x(h(t))=f(t), ~x(t)=0,~t\leq t_0,~\dot{x}(t)=0,~t<t_0
\end{equation}
is bounded on $[t_0,\infty)$ for any 
essentially bounded function $f$ on $[t_0,\infty)$. 
Then equation (\ref{1}) is exponentially stable.
\end{uess}

\begin{remark}
\label{remark1}
The Bohl-Perron principle is stated above for equation (\ref{1}) with 
two delays, but it is valid for linear equations
with an arbitrary number of delays, for integro-differential equations, and for equations with  
a distributed delay.
\end{remark}

\begin{remark}
\label{remark1a}
In Lemma~\ref{lemma3} we can consider
boundedness of solutions not for all essentially bounded functions $f$ on $[t_0,\infty)$  but only
for essentially bounded functions $f$ on $[t_1,\infty)$ that vanish on $[t_0,t_1)$ for any fixed $t_1>t_0$, see \cite{BB3}.
We further use this fact in the paper without an additional reference.
\end{remark}

Consider now a linear equation with a single delay and a non-negative coefficient
\begin{equation}\label{8}
\dot{x}(t)+a(t)x(h_0(t))=0, ~~a(t)\geq 0,~~ 0\leq t-h_0(t)\leq \tau_0
\end{equation}
and denote by $X_0(t,s)$ its fundamental function.

\begin{uess}\label{lemma4}\cite{BB3}
Assume that $X_0(t,s)>0$ , $t\geq s\geq t_0$. Then 
$$
\int_{t_0+\tau_0}^t X_0(t,s) a(s)ds\leq 1.
$$
\end{uess}

\begin{uess}\label{lemma5}\cite{BB3,GL}
If for some $t_0\geq 0$
$$
 \int_{h_0(t)}^t a(s) ds\leq \frac{1}{e},~t\geq t_0
$$
then $X_0(t,s)>0$ for $t\geq s\geq t_0$.

If in addition $a(t)\geq a_0>0$ then equation (\ref{8}) is 
exponentially stable.
\end{uess}

To extend stability results obtained for equation (\ref{1}) to other classes of equations,
we will need  the following three ``transformation" results reducing different classes of delay equations to an equation with a single delay.

\begin{uess}\label{lemma6}\cite[Lemma 5]{BB1_2006}
Assume that $a_k(t)\geq 0$, $h_k(t)\leq t$, $k=1,\dots, m$ are measurable, and $y$ is continuous on $[t_0, \infty)$. Then there exists a measurable function 
$h_0$ satisfying $$h_0(t)\leq t, ~~\min_k h_k(t)\leq h_0(t)\leq \max_k h_k(t)$$ 
such that $$\sum_{k=1}^m a_k(t) y(h_k(t))=\left(\sum_{k=1}^m a_k(t)\right) y(h_0(t)).$$
\end{uess}

\begin{uess}\label{lemma7}\cite[Corollary 9]{MCM2008}
Assume that $B(t,s)$ is a measurable non-decreasing in $s$ function,  
$h(t)\leq t$ is measurable, and $y$ is continuous  on $[t_0, \infty)$. 
Then there exists a measurable function 
$h_0$, $h(t)\leq h_0(t)\leq t$
such that $$\int_{h(t)}^t y(s)d_s B(t,s)  =\left(\int_{h(t)}^t d_s B(t,s)\right) y(h_0(t)).$$
\end{uess}

As a particular case of Lemma~\ref{lemma7}, we obtain the following result.

\begin{uess}\label{lemma8} 
Assume that $A(t,s)\geq 0$ is locally integrable, $h(t)\leq t$ is measurable, and $y$ is continuous  on $[t_0, \infty)$. 
Then there exists a measurable function $h_0$, $h(t)\leq 
h_0(t)\leq t$ such that $$\int_{h(t)}^t 
A(t,s)y(s)ds =\left(\int_{h(t)}^tA(t,s)ds \right)y(h_0(t)).$$
\end{uess}

\section{Main Results}

Let us fix a bounded interval $I=[t_0,t_1]$, $t_1>t_0\geq 0$, and for any essentially  bounded function on 
$[t_0,\infty)$, denote
$|f|_I=\esssup_{t\in I} |f(t)|$ for $I$ bounded and  
$\|f\|_{[t_0,\infty)}=\esssup_{t\geq t_0} |f(t)|$ for an unbounded interval.

Consider now initial value problem (\ref{10})
with  $\|f\|_{[t_0,\infty)}<\infty$. We have the following a~priori estimation.

\begin{uess}\label{lemma9}
Suppose (a1)-(a4) hold.
The solution of (\ref{10}) satisfies 
$$|\dot{x}|_I\leq \frac{\|b\|_{[t_0,\infty)}}{1-\|a\|_{[t_0,\infty)}}|x|_I+M_1,
\mbox{  where  } I=[t_0,t_1],~t_1>t_0\geq 0, ~M_1=\frac{\|f\|_{[t_0,\infty)}}{1-\|a\|_{[t_0,\infty)}} \, .
$$
\end{uess}
\begin{proof}
We have for $t\in I$
\begin{eqnarray*}
|\dot{x}(t)| & \leq & |a(t)||\dot{x}(g(t))|+|b(t)||x(h(t))|+\|f\|_{[t_0,\infty)}
\\
&\leq & \|a\|_{[t_0,\infty)}|\dot{x}|_I+\|b\|_{[t_0,\infty)}|x|_I+\|f\|_{[t_0,\infty)}.
\end{eqnarray*}
Hence
$\displaystyle
|\dot{x}|_I\leq \|a\|_{[t_0,\infty)}|\dot{x}|_I+\|b\|_{[t_0,\infty)}|x|_I+\|f\|_{[t_0,\infty)}$.
By (a2), $\|a\|_{[t_0,\infty)}<1$, thus the above inequality implies the estimate in the statement of the lemma.
\end{proof}

\begin {guess}\label{theorem1}
Assume that (a1)-(a4),(a6) hold and 
there exists $t_0\geq 0$ such that for $t\geq t_0$
$$
0<b_0\leq b(t),~~  \int_{h(t)}^t b(s)~ds \leq \frac{1}{e}\,,
$$
and
\begin{equation}\label{11}
\|a\|_{[t_0,\infty)}+\|b\|_{[t_0,\infty)}\left\|\frac{a}{b}\right\|_{[t_0,\infty)}<1.
\end{equation}
Then equation (\ref{1}) is exponentially stable.
\end{guess}
\begin{proof}
We will prove that the solution of (\ref{10}) for any $\|f\|_{[t_0,\infty)}<\infty$ is  bounded on $[t_0,\infty)$. 
Denote by $X_1(t,s)$ the fundamental function of the delay differential equation
\begin{equation}\label{13}
\dot{x}(t)+b(t)x(h(t))=0.
\end{equation}

Since $\displaystyle \int_{h(t)}^t b(s)ds 
\leq \frac{1}{e}$, by Lemma~\ref{lemma5}, $X_1(t,s)>0$ for any  
$t\geq s\geq t_0$.
The condition $b(t)\geq b_0>0$ and Lemma~\ref{lemma5} imply that equation (\ref{13}) is exponentially stable and 
$X_1(t,s)$ has an exponential estimate.

For a solution of  (\ref{10}) written in the form
$$
\dot{x}(t)+b(t)x(h(t))= a(t)\dot{x}(g(t))+f(t), ~x(t)=0, ~t\leq t_0, ~~\dot{x}(t)=0, ~t < t_0, 
$$
we have by Lemma \ref{lemma2} the representation 
$$
x(t)=\int_{t_0}^t X_1(t,s)a(s) \dot{x}(g(s))ds+f_1(t),
$$
where
$
f_1(t)=\int_{t_0}^t X_1(t,s) f(s)ds.
$
Since  $X_1(t,s)$ has an exponential estimate and $f$ is bounded on $[t_0,\infty)$, $\|f_1\|_{[t_0,\infty)}<\infty$.

Denote $I=[t_0,t_1]$. By Lemma \ref{lemma4},
\begin{eqnarray*}
|x(t)| & \leq & \int_{t_0}^t X_1(t,s)b(s)\left|\frac{a(s)}{b(s)}\right| |\dot{x}(g(s))|ds+|f_1(t)|
\\
& \leq & \left\|\frac{a}{b}\right\|_{[t_0,\infty)}|\dot{x}|_I+\|f_1\|_{[t_0,\infty)}.
\end{eqnarray*}
Hence
$$
|x|_I\leq \left\|\frac{a}{b}\right\|_{[t_0,\infty)}|\dot{x}|_I+\|f_1\|_{[t_0,\infty)}.
$$
Lemma \ref{lemma9} implies
\begin{eqnarray*}
|x|_I & \leq & \left\|\frac{a}{b}\right\|_{[t_0,\infty)}\left(\frac{\|b\|_{[t_0,\infty)}}{1-\|a\|_{[t_0,\infty)}}|x|_I+M_1\right)
+\|f_1\|_{[t_0,\infty)}
\\
 & = & \left\|\frac{a}{b}\right\|_{[t_0,\infty)}\frac{\|b\|_{[t_0,\infty)}}{1-\|a\|_{[t_0,\infty)}}|x|_I+M_2,
~~ M_2 := \left\|\frac{a}{b}\right\|_{[t_0,\infty)} M_1+ \|f_1\|_{[t_0,\infty)}.
\end{eqnarray*}

By (\ref{11}) we have $|x|_I\leq M$, where $M$ does not depend on the interval $I$. Hence $|x(t)|\leq M$ for $t 
\geq t_0$.
By Lemma \ref{lemma3}, equation (\ref{1}) is exponentially stable.
\end{proof}

\begin{corollary}\label{corollary1}
Assume that (a1)-(a4),(a6) hold and 
$$
b(t)\equiv b>0, ~~ b\tau \leq \frac{1}{e} \, , ~~ \|a\|_{[t_0,\infty)}<\frac{1}{2} \, .
$$
Then equation (\ref{1}) is exponentially stable.
\end{corollary}

Further, we denote $u^+=\max\{ u,0 \}$. 

\begin{guess}
\label{theorem2}
Assume that (a1)-(a4),(a6) are satisfied,  $b(t) \geq \beta>0$ and at least one of the following conditions holds

a) 
\begin{equation}\label{a}
\left\|\frac{a}{b_0}\right\|_{[t_0,\infty)}\frac{\|b\|_{[t_0,\infty)}}{1-\|a\|_{[t_0,\infty)}}
+\left\|\frac{b-b_0}{b_0}\right\|_{[t_0,\infty)} < 1,
\end{equation}
where $\displaystyle b_0(t)= \min \left\{ b(t),\frac{1}{\tau e}\right\}$;

b)
\begin{equation}\label{b}
 \displaystyle \| b\|_{[t_0,\infty)} \left(  \left\|\frac{a}{b}\right\|_{[t_0,\infty)}
 + \left\| \left( t-h(t)- \frac{1}{\|b\|_{[t_0,\infty)}e} \right)^+ \right\|_{[t_0,\infty)} \right) < 1- \| a \|_{[t_0,\infty)}.
\end{equation}
Then equation (\ref{1}) is exponentially stable.
\end{guess}
\begin{proof}
Assume that the condition in a) holds  and consider problem (\ref{10}) with 
$\|f\|_{[t_0,\infty)}<\infty$.  
We have $\displaystyle b_0(t)= \min \left\{ b(t),\frac{1}{\tau e}  \right\}  \geq \beta_1 := \min\left\{ \beta, \frac{1}{\tau e}  \right\}>0$.
Then $\displaystyle \int_{h(t)}^t b_0(s)ds\leq\frac{1}{e}$ and $0<\beta_1 \leq b_0(t)\leq b(t)$.
Problem (\ref{10}) can be rewritten as
$$
\dot{x}(t)+b_0(t)x(h(t))=a(t)\dot{x}(g(t))-(b(t)-b_0(t))x(h(t))+f(t),$$ $$x(t)=0,t\leq t_0,~\dot{x}(t)=0, t< t_0.$$
Denote by $X_2(t,s)$ the fundamental function of the equation
\begin{equation}\label{14}
\dot{x}(t)+b_0(t)x(h(t))=0.
\end{equation}
By Lemma \ref{lemma5},  $X_2(t,s)>0$ and equation (\ref{14}) is exponentially stable.

Let $I=[t_0,t_1]$. For the solution of  (\ref{10}) we have 
$$
x(t)=\int_{t_0}^t X_2(t,s) \left[a(s)\dot{x}(g(s))-(b(s)-b_0(s))x (h(s)) \right]ds+f_1(t),
$$
where $f_1(t)=\int_{t_0}^t X_2(t,s)f(s)ds$ and $\|f_1\|_{[t_0,\infty)}<\infty$.
Then \\
$\displaystyle
|x(t)|\leq \int\limits_{t_0}^t  X_2(t,s)b_0(s)\left[ \left. \left. \frac{1}{b_0(s)}\right( |a(s)||\dot{x}(g(s))|+(b(s)-b_0(s))|x (h(s)) |
\right)\right]ds$ \\ $\displaystyle +\|f_1\|_{[t_0,\infty)}$.
Hence, first by Lemma~\ref{lemma4} and then by Lemma~\ref{lemma9},
\begin{eqnarray*}
|x|_I & \leq &  \left\|\frac{a}{b_0}\right\|_{[t_0,\infty)}|\dot{x}|_I+ 
\left\|\frac{b-b_0}{b_0}\right\|_{[t_0,\infty)}|x|_I+\|f_1\|_{[t_0,\infty)}
\\
& \leq & \left(\left\|\frac{a}{b_0}\right\|_{[t_0,\infty)} \frac{\|b\|_{[t_0,\infty)}}{1-\|a\|_{[t_0,\infty)}}
+\left\|\frac{b-b_0}{b_0}\right\|_{[t_0,\infty)}\right) |x|_I +M_1.
\end{eqnarray*}
Condition (\ref{a}) 
implies $|x|_I<M$, where $M$ does not depend on the interval $I$. Hence $\|x\|_{[t_0,\infty)}<\infty$,
and therefore by Lemma \ref{lemma3} equation~(\ref{1}) is exponentially stable.

b) 
Denote $\displaystyle 
h_0(t)=\max\left\{ h(t), t-\frac{1}{\|b\|_{[t_0,\infty)}e}  \right\}$.
Then 
$\displaystyle \int_{h_0(t)}^t b(s)ds \leq \frac{1}{e}$, \\ $h_0(t)\geq h(t)$ and 
$\displaystyle |h(t)-h_0(t)|=\left( t-h(t)- \frac{1}{\|b\|_{[t_0,\infty)}e} \right)^+$.


Problem  (\ref{10}) can be rewritten as
$$
\dot{x}(t)+b(t)x(h_0(t))=a(t)\dot{x}(g(t))+b(t)\int_{h(t)}^{h_0(t)}\dot{x}(s)ds+f(t),~x(t)=\dot{x}(t)=0,~ t\leq t_0.
$$
Denote by $X_3(t,s)$ the fundamental function of the equation
\begin{equation}\label{17a}
\dot{x}(t)+b(t)x(h_0(t))=0.
\end{equation}
By Lemma~\ref{lemma5},   $X_3(t,s)>0$ and equation (\ref{17a}) is exponentially stable.

Denote $I=[t_0,t_1]$. We have
$$
x(t)=\int_{t_0}^t X_3(t,s)b(s)\left[\frac{1}{b(s)}\left(a(s)\dot{x}(g(s))+b(s)\int_{h(s)}^{h_0(s)}\dot{x}(\xi)d\xi\right)\right]
ds+f_3(s),
$$
where $f_3(t)=\int_{t_0}^t X_3(t,s)f(s)ds$ and $\|f_3\|_{[t_0,\infty)}<\infty$. 
Therefore Lemma~\ref{lemma9} yields that
\begin{eqnarray*}
|x|_I &\leq & \left\| \frac{a}{b} \right\|_{[t_0,\infty)}|\dot{x}|_I+\|h_0-h\|_{[t_0,\infty)}|\dot{x}|_I+\|f_3\|_{[t_0,\infty)} \\
& \leq &  \left( \left\|\frac{a}{b}\right\|_{[t_0,\infty)}+ \left\| \left( t-h(t)- \frac{1}{\|b\|_{[t_0,\infty)}e} \right)^+ \right\|_{[t_0,\infty)} \right)
\frac{\|b\|_{[t_0,\infty)}}{1-\|a\|_{[t_0,\infty)}} |x|_I+M_2.
\end{eqnarray*}

Inequality (\ref{b}) implies 
$\|x\|_{[t_0,\infty)}\leq M$, where $M$ does not depend on the interval $I$, thus $\|x\|_{[t_0,\infty)}<\infty$,
and therefore equation (\ref{1}) is exponentially stable.
\end{proof}

\begin{corollary}\label{corollary2a}
Assume that (a1)-(a4),(a6) are satisfied and at least one of the following conditions holds for $t \geq t_0$: 
\vspace{2mm}

a) $\displaystyle b(t) \geq \frac{1}{\tau e}$ and
$\displaystyle \tau\|b\|_{[t_0,\infty)}< \left.\left. \frac{2}{e}\right( 1-\|a\|_{[t_0,\infty)} \right)$;
\vspace{2mm}

b) $b(t)\geq \beta >0, \displaystyle t-h(t)\geq \frac{1}{\|b\|_{[t_0,\infty)}e}$,  \\ $\displaystyle \left\|\frac{a}{b}\right\|_{[t_0,\infty)}\|b\|_{[t_0,\infty)}+\tau \|b\|_{[t_0,\infty)}<
1+\frac{1}{e}-\|a\|_{[t_0,\infty)}$. 
\vspace{2mm}

Then equation (\ref{1}) is exponentially stable.
\end{corollary}

\begin{proof}
Conditions in
a)  of the corollary yield that $b_0(t)=\frac{1}{\tau e}$ and 
$\displaystyle
\left\|b-b_0 \right\|_{[t_0,\infty)}=
\|b\|_{[t_0,\infty)}-\frac{1}{\tau e}.
$
Hence, after some simple calculations,  condition a) of the corollary implies (\ref{a}) of Theorem~\ref{theorem2}.

Assume that $t-h(t)\geq \frac{1}{\|b\|_{[t_0,\infty)}e}$. Then
\begin{eqnarray*}
 \left\| \left( t-h(t)- \frac{1}{\|b\|_{[t_0,\infty)}e} \right)^+ \right\|_{[t_0,\infty)} &= &
 \left\|  t-h(t)- \frac{1}{\|b\|_{[t_0,\infty)}e}\right\|_{[t_0,\infty)}
\\
& = & \|t-h(t)\|_{[t_0,\infty)}-\frac{1}{\|b\|_{[t_0,\infty)}e} \\ & \leq & \tau-\frac{1}{\|b\|_{[t_0,\infty)}e}.
\end{eqnarray*}
The inequality
$\displaystyle
 \left( \left\|\frac{a}{b}\right\|_{[t_0,\infty)}+ \tau-\frac{1}{\|b\|_{[t_0,\infty)}e}\right)
\frac{\|b\|_{[t_0,\infty)}}{1-\|a\|_{[t_0,\infty)}}<1 
$
is equivalent to the last inequality in b).
\end{proof}


Considering $b(t) \equiv b$ with the cases $t-h(t) \geq \frac{1}{eb}$ and $b>\frac{1}{\tau e}$ only, we get the following 
result.

\begin{corollary}\label{corollary2b}
Assume that (a1)-(a4),(a6) are satisfied, $b(t)\equiv b>0$,  and at least one of the following conditions holds for $t \geq t_0$:
\begin{equation}\label{A}
\frac{1}{e}\leq b\tau<\frac{2}{e}(1-\|a\|_{[t_0,\infty)});
\end{equation}
\begin{equation}\label{B}
\frac{1}{e}\leq b(t-h(t)) \leq b\tau <1+\frac{1}{e}-2\|a\|_{[t_0,\infty)}.
\end{equation}
Then equation (\ref{1}) is exponentially stable.
\end{corollary}

In the following theorem, for equation (\ref{1}) we obtain integral stability conditions  which do not assume boundedness of delays.
Denote for $b(t)\neq 0$ almost everywhere
\begin{equation} \label{4.9star}
A(t):=\frac{a(t)b(g(t))}{b(t)}.
\end{equation}

\begin{guess}\label{theorem2a}
Assume that (a1)-(a4) hold, $b(t)\geq 0$, $\int_0^{\infty} b(s) ds=\infty$, $b(t)\neq 0$ almost everywhere,
\begin{equation}\label{1a}
\limsup_{t\rightarrow\infty} \int_{g(t)}^t b(\xi)d\xi<\infty
\end{equation}
and at least one of the following conditions holds for $t \geq t_0$:

a) 
$\displaystyle
\int_{h(t)}^t b(\xi)d\xi\leq \frac{1}{e} \, ,~ \|A\|_{[t_0,\infty)}<\frac{1}{2};
$

b)
$\displaystyle
\frac{1}{e}< \int_{h(t)}^t b(\xi)d\xi< 1+\frac{1}{e} -2\|A\|_{[t_0,\infty)}.  
$

Then equation (\ref{1}) is asymptotically stable.
\end{guess}
\begin{proof}
Let ${\displaystyle s=p(t):=\int_{t_0}^t b(\tau)d\tau,~ y(s)=x(t)}$,
where $p(t)$ is a strictly increasing function.
Then we introduce $\tilde{h}(s)$ and $\tilde{g}(s)$
as follows:
$$
x(h(t))=y(\tilde{h}(s)), ~\tilde{h}(s)\leq s, ~ \tilde{h}(s)=\int_{t_0}^{h(t)} 
b(\tau)d\tau, ~ s-\tilde{h}(s)=\int_{h(t)} ^t b(\tau)d\tau,
$$ 
$$
\tilde{g}(s)=\int_{t_0}^{g(t)} b(\tau)d\tau,~ s-\tilde{g}(s)=\int_{g(t)} ^t b(\tau)d\tau, ~\tilde{g}(s)\leq s,
$$
$$
\dot{x}(t)=b(t)\dot{y}(s),~~ \dot{x}(g(t))=b(g(t))\dot{y}(\tilde{g}(s)).
$$
Equation (\ref{1}) can be rewritten in the form
\begin{equation}\label{2a}
\dot{y}(s)-\tilde{a}(s)\dot{y}(\tilde{g}(s))=-y(\tilde{h}(s)),
\end{equation}
where $\tilde{a}(s)=A(t)$, and $A$ is defined in (\ref{4.9star}). By inequalities (\ref{1a}), equation  (\ref{2a}) involves bounded delays. 
If $x(t)$ is a solution of  (\ref{1}) then $y(s)=x(t)$ is a solution of  (\ref{4.9star}). 

Corollary \ref{corollary1} and condition a) of the theorem, Corollary \ref{corollary2b} and condition b) of the theorem 
imply that 
equation (\ref{2a}) is exponentially stable.
Hence (\ref{1}) is stable and $\displaystyle \lim_{s\rightarrow\infty} y(s)=\lim_{t\rightarrow\infty} x(t)=0$,
i.e. (\ref{1}) is asymptotically stable.
\end{proof}

Theorem \ref{theorem2a} can be applied to derive 
stability conditions for pantograph-type neutral  equations with unbounded delays.

Consider the  equation
\begin{equation}\label{3a}
\dot{x}(t)-a\dot{x}(\mu t) = -\frac{b}{t}x(\lambda t), ~t\geq 1,~|a|<1,~b>0,~\mu \in (0,1), ~\lambda \in (0,1).
\end{equation}
Here
$$
\int_{\lambda t}^t \frac{b}{s} \, ds=b\ln \frac{1}{\lambda}<+\infty, ~~
\int_{\mu t}^t \frac{b}{s} \, ds=b\ln \frac{1}{\mu}< +\infty.
$$

\begin{corollary}\label{corollary3}
Suppose at least one of the following conditions holds:
\vspace{2mm}

a) $\displaystyle b \ln \frac{1}{\lambda}\leq \frac{1}{e}, ~ |a|< \frac{1}{2}$;
\vspace{2mm}

b) $\displaystyle \frac{1}{e}< b \ln \frac{1}{\lambda}<  1+\frac{1}{e}- 2|a|$.
\vspace{2mm}

Then equation (\ref{3a}) is asymptotically stable.
\end{corollary}

\section{Generalizations of Main Results}

\subsection{Equations with several delays}

We consider here an equation with several neutral terms
\begin{equation}
\label{16} 
\dot{x}(t)-\sum_{k=1}^{\ell} a_k(t)\dot{x}(g_k(t))=-b(t)x(h(t)),
\end{equation}
as well as an equation with several delayed terms (not including the derivative)
\begin{equation}
\label{17} 
\dot{x}(t)-a(t)\dot{x}(g(t))=-\sum_{k=1}^m b_k(t)x(h_k(t))
\end{equation}
under the following conditions:\\
(b1) $a, b, a_k, b_k,  g, h,  g_k, h_k$ are Lebesgue measurable  essentially
bounded functions on $[0,\infty)$;\\
(b2) $ \mbox{ess}\sup_{t\geq t_0}  |a(t)|\leq a_0<1$,  $  \sum_{k=1}^{\ell} \mbox{ess}\sup_{t\geq t_0} |a_k(t)|\leq A_0<1$, $b(t) \geq 0$, 
$b_k(t)\geq 0$, $t \geq t_0$ for some $t_0\geq 0$;\\ 
(b3) $0\leq t-g(t)\leq \delta,0\leq t-g_k(t)\leq \delta_k,~ mes~ E=0\Longrightarrow mes~ g^{-1}(E)=0, ~~mes~ g_k^{-1}(E)=0$; 
\\
(b4) $0\leq t-h(t)\leq \tau, 0\leq t-h_k(t)\leq \tau_k$.

\begin{guess}\label{theorem3}
Suppose the conditions of  Theorem \ref{theorem1} or at least one of the conditions a) or b)  of  Theorem \ref{theorem2}
hold, where the number $\|a\|_{[t_0,\infty)}$ is replaced by $\displaystyle \sum_{k=1}^{\ell} \|a_k\|_{[t_0,\infty)}$.
Then equation (\ref{16}) is exponentially stable.
\end{guess}
The proof follows the scheme of the proofs of Theorems \ref{theorem1} and \ref{theorem2}.

\begin{guess}\label{theorem4}
Suppose $b_k(t)\geq 0$, $t\geq t_0\geq 0$, conditions of  Theorem \ref{theorem1} or at least 
one of the conditions a) or b)  of  Theorem \ref{theorem2} hold, where $b, h,\tau$ are replaced by $\bar{b}, \bar{h}, \bar{\tau}$, 
$\displaystyle \bar{b}(t):=\sum_{k=1}^m b_k(t)$, $\displaystyle \bar{h}(t):=\min_k h_k(t), \bar{\tau}:=\max_k\tau_k $.
Then equation (\ref{17}) is exponentially stable.
\end{guess}
\begin{proof}
Assume that the conditions of Theorem \ref{theorem1} hold, where $b, h, \tau$ are replaced by $\bar{b}, \bar{h}, \bar{\tau}$.
Consider the initial value problem
\begin{equation}
\label{19}
\dot{x}(t)-a(t)\dot{x}(g(t))=-\sum_{k=1}^m b_k(t)x(h_k(t))+f(t), ~x(t)=\dot{x}(t)=0,~~ t\geq t_0,
\end{equation}
with $\|f\|_{[t_0,\infty)}<\infty$, where $x$ is a solution of (\ref{19}). By Lemma~\ref{lemma6},
there exists $h_0(t)$, $\bar{h}(t)\leq h_0(t)\leq t$ such that $\sum_{k=1}^m b_k(t)x(h_k(t))=\bar{b}(t)x(h_0(t))$.
Hence $x$ is a solution of the initial value problem
\begin{equation}
\label{20}
\dot{x}(t)-a(t)\dot{x}(g(t))=-\bar{b}(t)x(h_0(t))+f(t), ~x(t)=\dot{x}(t)=0,~~ t\leq t_0.
\end{equation}
Since $\bar{h}(t)\leq h_0(t)$ and $t-h_0(t)\leq \bar{\tau}$, all conditions of Theorem \ref{theorem1} hold, 
where $b,h,\tau$ are replaced by $\bar{b},h_0,\bar{\tau}$.
By Theorem~\ref{theorem1}, the equation 
$$
\dot{y}(t)-a(t)\dot{y}(g(t))=-\bar{b}(t)y(h_0(t))
$$
is exponentially stable. Therefore the solution $x$ of (\ref{20}) is a bounded function.
By Lemma \ref{lemma3}, equation (\ref{17}) is exponentially stable.

The proof of the second part is similar.
\end{proof}

Theorem \ref{theorem4} uses the worst delay function $\bar{h}(t)\leq h_k(t)$, $k=1,\dots,m$. 
In the next theorem, we obtain sharper results by using all the delays.
Further, to simplify the notation, we will write $\| \cdot \|$ instead of $\| \cdot \|_{[t_0,\infty)}$.

\begin{guess}\label{theorem5}
Suppose there exist $t_0\geq 0$, $\beta>0$, and a set of indexes $J\subseteq \{1,\dots,m\}$ such that
 $\displaystyle b(t):=\sum_{k\in J} b_k(t)\geq \beta$,  for $t\geq t_0$ and
\begin{equation}\label{22}
\left(\left\|\frac{a}{b}\right\|+\sum_{k\in J} \tau_k\left\|\frac{b_k}{b}\right\|\right)
\frac{\sum_{k=1}^m \|b_k\|}{1-\|a\|}+\sum_{k \notin J}\left\|\frac{b_k}{b}\right\|<1.
\end{equation}
Then equation (\ref{17}) is exponentially stable.
\end{guess}

\begin{proof}
Consider initial value problem (\ref{19}) with $\|f\|<\infty$.
Equation (\ref{19}) can be rewritten in the form
$$
\dot{x}(t)+b(t)x(t)=a(t)\dot{x}(g(t))+\sum_{k\in J} b_k(t)\int_{h_k(t)}^t \dot{x}(s)ds+\sum_{k\notin J} b_k(t)x(h_k(t))+f(t).
$$
Hence
\begin{eqnarray*}
x(t)= & \int_{t_0}^t e^{-\int_s^t b(\xi)d\xi}b(s)\left[\frac{1}{b(s)}
\left( a(s)\dot{x}(g(s))+\sum_{k\in J} b_k(s)\int_{h_k(s)}^s  \dot{x}(\xi)d\xi
\right. \right.
\\
& \left. \left.+\sum_{k\notin J} b_k(s)x(h_k(s))\right) 
\right]ds+f_1(t),
\end{eqnarray*}
where $\displaystyle f_1(t)=\int_{t_0}^t e^{-\int_s^t b(\xi)d\xi}f(s)ds$, and $\|f_1\|<\infty$.

Denote $I=[t_0,t_1]$. Then
$$
|x|_I\leq \left(\left\|\frac{a}{b}\right\|+\sum_{k\in J} \tau_k\left\|\frac{b_k}{b}\right\|\right)|\dot{x}|_I
+\sum_{k \notin J}\left\|\frac{b_k}{b}\right\||x|_I+M_1.
$$
From (\ref{19}), similarly to the proof of Lemma \ref{lemma9}, we obtain an a~priori estimate
\begin{equation}
\label{add_star1}
|\dot{x}|_I\leq \frac{\sum_{k=1}^m \|b_k\|}{1-\|a\|}|x|_I+M_2.
\end{equation}
Therefore
$$
|x|_I\leq \left[\left(\left\|\frac{a}{b}\right\|+\sum_{k\in J} \tau_k\left\|\frac{b_k}{b}\right\|\right)
\frac{\sum_{k=1}^m \|b_k\|}{1-\|a\|}\right.
\left.+\sum_{k \notin J}\left\|\frac{b_k}{b}\right\|\right]|x|_I+M,
$$
where $M$ does not depend on the interval $I$. Inequality  (\ref{22}) implies $\|x\|<\infty$,
thus by Lemma \ref{lemma3} equation  (\ref{17}) is exponentially stable.
\end{proof}
\begin{corollary}\label{corollary4}
Suppose for  $t_0\geq 0$, 
 $\displaystyle b(t):=\sum_{k=1}^m b_k(t)\geq b_0>0$ and
\begin{equation}\label{22a}
\left(\left\|\frac{a}{b}\right\|+\sum_{k=1}^m \tau_k\left\|\frac{b_k}{b}\right\|\right)
\frac{\sum_{k=1}^m \|b_k\|}{1-\|a\|}<1.
\end{equation}
Then equation (\ref{17}) is exponentially stable.
\end{corollary}
\begin{proof}
The statement of the corollary follows from Theorem~\ref{theorem5} if we take $J=\{1,2,\dots,m\}$. 
\end{proof}
\begin{corollary}\label{corollary5}
Suppose there exists $t_0\geq 0$ and index $i, 1\leq i\leq m$ such that for $t\geq t_0$, $b_i(t)\geq \beta >0$ and
\begin{equation}\label{25}
\left(\left\|\frac{a}{b_i}\right\|+\tau_i\right)\frac{\sum_{k=1}^m \|b_k\|}{1-\|a\|}
+\sum_{k\neq i} \left\|\frac{b_k}{b_i}\right\|<1.
\end{equation}
Then equation (\ref{17}) is exponentially stable.
\end{corollary}
\begin{proof}
We take $J=\{i\}\subseteq \{1,2,\dots,m\}$ and apply Theorem~\ref{theorem5}. 
\end{proof}

In the next theorem we partially improve the results of the previous theorem.
\begin{guess}\label{theorem6}
Suppose there exist $t_0\geq 0$, $b_0>0$ and a set of indexes $J\subseteq \{1,\dots,m\}$ such that
$\displaystyle b(t):=\sum_{k\in J} b_k(t)\geq b_0$  for $t\geq t_0$ and also
\begin{equation}\label{32a}
\left(\left\|\frac{a}{b}\right\|
+\sum_{k\in J} \left\|\frac{b_k}{b}\right\|\left\|t-h_k(t)-\frac{1}{Be}\right\|\right)
\frac{\sum_{k=1}^m \|b_k\|}{1-\|a\|}+\sum_{k \notin J}\left\|\frac{b_k}{b}\right\|<1,
\end{equation}
where $B=\sum_{k\in J} \|b_k\|$.
Then equation (\ref{17}) is exponentially stable.
\end{guess}
\begin{proof}
Consider initial value problem (\ref{19}) with $\|f\|<\infty$.
Equation (\ref{19}) can be rewritten in the form
\begin{eqnarray*}
\dot{x}(t)+b(t)x\left(t-\frac{1}{Be}\right)
= & a(t)\dot{x}(g(t))+\sum_{k\in J} b_k(t)\int_{h_k(t)}^{t-\frac{1}{Be}} \dot{x}(s)ds
\\ & +\sum_{k\notin J}^m b_k(t)x(h_k(t))+f(t).
\end{eqnarray*}
Denote by $X_1(t,s)$ the fundamental function of the equation
\begin{equation}\label{17b}
\dot{x}(t)+b(t)x\left(t-\frac{1}{Be}\right)=0.
\end{equation}
By Lemma~\ref{lemma5},   $X_1(t,s)>0$ and equation (\ref{17b}) is exponentially stable.

Hence from (\ref{10})
\begin{eqnarray*}
x(t) & = & \int_{t_0}^t X_1(t,s)b(s)\left[\frac{1}{b(s)}
\left[a(s)\dot{x}(g(s))+\sum_{k\in J} b_k(s)\int_{h_k(s)}^{s-\frac{1}{Be}} \dot{x}(\xi)d\xi\right.\right.
\\
& & \left.\left.+\sum_{k\notin J}^m b_k(s)x(h_k(s))\right]\right)ds+f_1(t),
\end{eqnarray*}
where $f_1(t)=\int_{t_0}^t X_1(t,s)f(s)ds$, and $\|f_1\|<\infty$.
Denote $I=[t_0,t_1]$. Then
$$
|x|_I  \leq  \left(\left\|\frac{a}{b}\right\|
+\sum_{k\in J} \left\|\frac{b_k}{b}\right\|\left\|t-h_k(t)-\frac{1}{Be}\right\|\right)|\dot{x}|_I
+\sum_{k \notin J}\left\|\frac{b_k}{b}\right\||x|_I+M_1.
$$
From (\ref{19}), 
using a~priori estimate (\ref{add_star1}), we obtain
\begin{eqnarray*}
|x|_I  \leq & \left[\left(\left\|\frac{a}{b}\right\|
+\sum_{k\in J} \left\|\frac{b_k}{b}\right\|\left\|t-h_k(t)-\frac{1}{Be}\right\|\right)
\frac{\sum_{k=1}^m \|b_k\|}{1-\|a\|} \right.
\\ &
\left.  +\sum_{k \notin J}\left\|\frac{b_k}{b}\right\|\right]|x|_I+M,
\end{eqnarray*}
where $M$ does not depend on the interval $I$. Inequality  (\ref{32a}) implies $\|x\|<\infty,$
thus by Lemma \ref{lemma3} equation  (\ref{17}) is exponentially stable.
\end{proof}

\begin{corollary}\label{corollary6}
Suppose there exists $t_0\geq 0$ such that
 $\displaystyle b(t):=\sum_{k=1}^m b_k(t)\geq b_0>0$, $|a(t)|\leq a_0<1$ for $t\geq t_0$, and
\begin{equation}\label{23a}
\left(\left\|\frac{a}{b}\right\|+\sum_{k=1}^m \left\|\frac{b_k}{b}\right\|
\left\|t-h_k(t)-\frac{1}{\sum_{k=1}^m \|b_k\|e}\right\|\right)
\frac{\sum_{k=1}^m \|b_k\|}{1-\|a\|}<1.
\end{equation}
Then equation (\ref{17}) is exponentially stable.
\end{corollary}
\begin{corollary}\label{corollary7}
Suppose there exists $t_0\geq 0$ and index $i, 1\leq i\leq m$ such that $b_i(t)\geq b_0>0$ for $t\geq t_0$,
and the following condition holds
\begin{equation}\label{25a}
\left(\left\|\frac{a}{b_i}\right\|+\left\|t-h_i(t)
-\frac{1}{\|b_i\|e}\right\|\right)\frac{\sum_{k=1}^m \|b_k\|}{1-\|a\|}
+\sum_{k\neq i} \left\|\frac{b_k}{b_i}\right\|<1.
\end{equation}
Then equation (\ref{17}) is exponentially stable.
\end{corollary}

\begin{corollary}\label{corollary8}
Suppose there exists $t_0\geq 0$ such that
 $\displaystyle b(t):=\sum_{k=1}^m b_k(t)\geq b_0>0$,  $\displaystyle t-h_k(t)\geq \frac{1}{\sum_{k=1}^m \|b_k\|e} $ for $t\geq 
t_0$, $k=1, \dots, m$, and
\begin{equation}\label{24A}
\left(\left\|\frac{a}{b}\right\|+\sum_{k=1}^m \tau_k\left\|\frac{b_k}{b}\right\|\right)
\sum_{j=1}^m \|b_j\|<1+\frac{1}{e} \sum_{k=1}^m \left\|\frac{b_k}{b}\right\|-\|a\|.
\end{equation}
Then equation (\ref{17}) is exponentially stable.
\end{corollary}
\begin{proof}
We have for $t\geq t_0$ 
$$
\left|t-h_k(t)-\frac{1}{\sum_{j=1}^m \|b_j\|e}\right|=t-h_k(t)-\frac{1}{\sum_{j=1}^m \|b_j\|e}\leq \tau_k-\frac{1}{\sum_{j=1}^m 
\|b_j\|e}.
$$
Hence (\ref{23a}) holds if 
\begin{equation}\label{24B}
\left[\left\|\frac{a}{b}\right\|+\sum_{k=1}^m \left\|\frac{b_k}{b}\right\|
\left(\tau_k-\frac{1}{\sum_{j=1}^m \|b_j\|e}\right)\right]\sum_{j=1}^m \|b_j\| < 1-\|a\|.
\end{equation}
Inequality (\ref{24B}) is equivalent to (\ref{24A}).
\end{proof} 

\subsection{Equations with distributed delays and integro-differential \\ equations}

Consider a neutral equation with distributed delays
\begin{equation}\label{27}
\dot{x}(t)- a(t) \dot{x}(g(t)) +b(t)\int_{h(t)}^t x(s) d_s B(t,s)=0,
\end{equation}
where $a$, $b$, $g$, $h$ satisfy (a1)-(a4), $B(t,s)$ is measurable on $[0,\infty)\times [0,\infty)$,
$B(t, \cdot)$ is a left continuous non-decreasing function
for almost all  $t$, $B(\cdot,s)$  is locally integrable for
any $s$, $B(t,h(t))=0$, and $B(t,t^+)=1$.
Then $\int_{h(t)}^t d_s B(t,s)=1$.

\begin{guess}\label{theorem7}
Suppose there exists $t_0\geq 0$ such that $a(t)\leq a_0<1$, $b(t)\geq b_0>0$, $t-g(t)\leq \delta$, $t-h(t)\leq \tau$  for $t \geq t_0$ and at least 
one of the conditions of Theorems~\ref{theorem1} or \ref{theorem2} hold. Then equation 
(\ref{27}) is exponentially stable.
\end{guess}
\begin{proof}
Suppose that for $t\geq t_0$, $x$ is a solution of the  initial value problem
\begin{equation}\label{28}
\dot{x}(t)- a(t) \dot{x}(g(t)) +b(t)\int_{h(t)}^t x(s) d_s B(t,s)=f(t),~~x(t)=\dot{x}(t)=0, ~t\leq t_0,
\end{equation}
where $f$ is an essentially bounded function on $[t_0,\infty)$. 
By Lemma~\ref{lemma7}  there exists a function $h_0(t)$,  $h(t)\leq h_0(t)\leq t$ such that
$$
 \int_{h(t)}^t  x(s)d_s B(t,s)= x(h_0(t)),
$$
hence $x$ satisfies the equation
\begin{equation}
\label{29}
\dot{x}(t)-a(t)x(g(t))+b(t)x(h_0(t))=f(t).
\end{equation}
By either Theorem \ref{theorem1} or \ref{theorem2}, equation (\ref{29}) is exponentially stable.
Hence $x$ is a bounded on $[t_0,\infty)$ function, therefore by Lemma \ref{lemma3} equation  (\ref{27}) is also exponentially 
stable.
\end{proof}

The integro-differential equation
\begin{equation}\label{30}
\dot{x}(t)-a(t)\dot{x}(g(t))+\int_{h(t)}^t K(t,s) x(s)ds=0,
\end{equation}
where $K(t,s)$ is a Lebesgue 
measurable locally integrable function on $[0,\infty)\times [0,\infty)$, \\ $K(t,s)\geq 0$, 
 is a particular case of (\ref{27}).
After denoting 
$$
b(t)=\int_{h(t)}^t K(t,s)ds,
$$
$$
B(t,s)= \left\{ \begin{array}{ll} \displaystyle  \frac{1}{b(t)} \int_{h(t)}^s K(t,\zeta)~d\zeta, & b(t)>0,
\\ 0, & b(t)=0,  \end{array} \right. 
$$
equation (\ref{30}) has the form of (\ref{27}). We assume that $a$, $b$, $g$, $h$ satisfy (a1)-(a4). 

As a direct corollary of Theorem~\ref{theorem7} we obtain the following result.
\begin{guess}\label{theorem8}
Suppose there exists $t_0\geq 0$ such that $a(t)\geq a_0>0$,
 $b(t):=\int_{h(t)}^t K(t,s)ds\geq b_0>0$, $t-g(t)\leq \delta$, $t-h(t)\leq \tau$ for $t \geq t_0$ and at least 
one of the conditions of Theorems~\ref{theorem1} or \ref{theorem2} hold. Then equation 
(\ref{30}) is exponentially stable.
\end{guess}

\section{Discussion and Open Problems}

In the present paper, we have not considered  mixed neutral differential equations
which include delay terms together with integro-differential 
or distributed delay terms,
for example, the equation
\begin{equation}\label{31}
\dot{x}(t)-a(t)\dot{x}(g(t))+ p(t)x(r(t))+\int_{h(t)}^t K(t,s) x(s)ds=0.
\end{equation}
However, if $x$ is a solution of (\ref{31}) and $K(t,s)\geq 0$ then, by Lemma \ref{lemma8},
there exists a function $h_0(t)$, $h(t)\leq h_0(t)\leq t$ such that $x$ is also
a solution of the equation
\begin{equation}\label{32}
\dot{x}(t)-a(t)\dot{x}(g(t))+ p(t)x(r(t))+b(t)x(h_0(t))=0,
\end{equation}
where $b(t)=\int_{h(t)}^t K(t,s) ds$. Hence any stability test for
equation (\ref{32}) with three delays implies stability conditions for
mixed neutral  differential equation (\ref{31}).

Another possible extension includes a distributed delay in the derivative part, such as 
$\displaystyle \dot{x}(t) -\int_{g(t)}^t\dot{x}(s)d_s A(t,s)$ or $\displaystyle \dot{x}(t) -\int_{g(t)}^t  A(t,s)\dot{x}(s)ds$.
Applying either Lemma~\ref{lemma7} or  \ref{lemma8}, we can transform such differential equations 
to equations considered in the paper. 

Theorem~\ref{theorem2a} allows us to obtain asymptotic stability conditions for equation (\ref{1})
with unbounded delays. The same approach can be applied to other neutral equations considered in the paper.

Let us discuss now both known results and new stability tests presented in the paper.

Proposition~\ref{proposition1} has a simple form but involves several unnecessary restrictions, such as 
$c(t)\geq 0$ and differentiability of $c$.

Propositions~\ref{proposition2} and \ref{proposition2a} in the non-neutral case $p(t)\equiv 0$ are reduced to the best possible
asymptotic stability condition $\displaystyle \limsup_{t\rightarrow\infty}\int_{t-\sigma}^t Q(s)ds<\frac{3}{2}$.
For equation (\ref{1})  such results are unknown.
We recall that these three propositions can only be applied to equation  (\ref{1}) if $a(t)\equiv a$, $g(t)=t-\sigma$.

Proposition \ref{proposition3} gives explicit stability conditions for a general autonomous neutral equation  
which coincide with known stability results for equations without the neutral part.

Proposition \ref{proposition4} presents stability conditions in an integral form which is explicit but a little bit artificial. 

Proposition \ref{proposition5} is an extension of a well known stability result to neutral equations: if the equation 
includes a non-delay term and this term dominates over the other terms, this equation is asymptotically stable.
In our opinion, Proposition \ref{proposition5} is one of the best stability results for neutral equations.
It is interesting whether the statement remains true for equations with variable delays.

Most of previous results were obtained for equations with constant delays. The fixed point method gives an opportunity to consider equations with variable delays. One of such typical results is given in Proposition~\ref{proposition6}.
However, this statement has some unnecessary restrictions: the equation should include a non-delay term, the delay function
has to be twice differentiable, and the delays in the neutral and the non-neutral parts coincide.

The result of Proposition~\ref{proposition7}, where the author studies an asymptotic property different from the 
asymptotic stability, is interesting due to the method applied.  
In particular, the author used estimates of the fundamental
function for a non-oscillatory equation.

In all stability results of Propositions~\ref{proposition1}-\ref{proposition7}, it was assumed that 
all parameters of considered neutral equations are continuous functions,
and the proofs were based on this assumption. Hence all the results are not available for equations
with measurable parameters. Equations with measurable parameters and with variable delays were considered
in Propositions \ref{proposition8}-\ref{proposition9}. The method applied there was based on the Bohl-Perron theorem. 
Proposition~\ref{proposition8} has a simple form, but also includes some restrictions on the coefficients that we omit in the results of the present paper.
In particular, it is assumed that 
the sum of the coefficients of the neutral terms is non-negative.
In Proposition~\ref{proposition9} stability results depend on the delays in the neutral terms. 

The stability tests obtained in the present paper are also based on the Bohl-Perron theorem. 
However, we use a different approach applying a~priori estimates of solutions and integral
inequalities for fundamental functions of non-oscillatory delay differential equations.
The results are also different from those in Propositions~\ref{proposition8}-\ref{proposition9}. 
In particular, we do not assume that the sum of the coefficients of the neutral terms is non-negative
as in Proposition \ref{proposition8}. Our conditions do not depend on the delay in the neutral terms as in
Proposition \ref{proposition9}. Thus our conditions and conditions of Propositions \ref{proposition8}-\ref{proposition9}
are independent.

Let us emphasize that Theorem \ref{theorem2} for the non-neutral case implies the best possible known stability
condition $\tau \|b\|_{[t_0,\infty)}<1+\frac{1}{e}$ for delay differential equations with one delay and measurable parameters.
Theorems~\ref{theorem5} and \ref{theorem6} for neutral equations with several delays give a set of $2^m-1$ different
stability conditions. Some of these conditions are presented as corollaries to these theorems.

\begin{example}\label{example1}
Let us compare results considered in the paper, both previously known and new. To this end consider the autonomous 
\begin{equation}\label{33a}
\dot{x}(t)-\frac{1}{3}\dot{x}(t-\sigma)+\frac{1}{3}x(t-\tau)=0
\end{equation}
and the non-autonomous equations
\begin{equation}\label{33b}
\dot{x}(t)-\frac{1}{3}\dot{x}(g(t))+\frac{1}{3}x(t-\tau)=0, ~~0 \leq t-g(t) \leq \sigma. 
\end{equation} 
We will find conditions on $\tau\geq 0$ such that for all $\sigma\geq 0$ equation (\ref{33a}) is exponentially stable.
Since the equation must be exponentially stable for $\sigma=0$, we have the necessary exponential stability condition $\tau\leq \pi$.

Proposition~\ref{proposition1} is not applicable, since the coefficient of the neutral term in (\ref{33a}) is negative.


By Proposition \ref{proposition2},  the inequality $\tau <\frac{7}{6}$ implies exponential stability.

By Proposition~\ref{proposition2a}, any $\tau <\sqrt{6}\approx 2.45$ is appropriate.

Propositions~\ref{proposition3}, \ref{proposition5}, \ref{proposition6} do not work 
since the equations should include a non-delay term.

By Proposition \ref{proposition4}, $\tau < \frac{5}{9}$ is sufficient.

Proposition~\ref{proposition7} does not establish exponential stability.

By Proposition \ref{proposition8}, we have $\tau<\frac{3}{e}$.

Propositions \ref{proposition9} is not applicable to this equation, since
it does not establish exponential stability for any $\sigma$ in the neutral term.

By Corollary \ref{corollary1} and inequality (\ref{B}) in Corollary~\ref{corollary2b},  
$\displaystyle \tau<\frac{3}{e}+1\approx 2.1$ implies exponential stability.

Thus the best condition for autonomous equation (\ref{33a}) is due to  Proposition~\ref{proposition2a}. 

For non-autonomous equation (\ref{33b}), the situation is more complicated, since 
only Proposition \ref{proposition8}, Corollary \ref{corollary1} and Corollary \ref{corollary2b} are applicable.

By Proposition \ref{proposition8} we have $\tau<\frac{3}{e}$. By Corollary \ref{corollary1}, together with 
inequality (\ref{B}) in Corollary~\ref{corollary2b}, we obtain  $\tau<1+\frac{3}{e}$. 
\end{example}

\begin{example}\label{example2}
Consider again equations  (\ref{33a}) and  (\ref{33b}) with a fixed $\sigma>0$, $g(t) \equiv t-\sigma$
and compare Proposition~\ref{proposition9} and  Corollary~\ref{corollary2b}.

By (\ref{add_star}) in Corollary~\ref{corollary01} of Proposition~\ref{proposition9} we have the 
exponential stability condition 
$$
\frac{2}{e} -1+ \frac{\sigma}{2} < \tau < 1+\frac{2}{e}-\frac{\sigma}{2},
$$
while by Corollaries  \ref{corollary1} and \ref{corollary2b} we obtain the condition 
$$
\tau<1+\frac{3}{e},
$$
which is better than $\displaystyle \tau < 1+\frac{2}{e}-\frac{\sigma}{2}$ for any $\sigma \geq  0$.
\end{example}

Let us outline some open problems and topics for future research.

\begin{enumerate}

\item
Is it possible in Corollaries  \ref{corollary1}  to improve the condition $\|a\|_{[t_0,\infty)}<\frac{1}{2}$ 
 to the condition $\|a\|_{[t_0,\infty)}<\lambda$, where $\lambda\in 
(\frac{1}{2},1)$?

What is the best possible $\lambda$ in this condition? 

\item
Is it possible to apply the method reducing a neutral equation to an equation with an infinite number of delays
earlier used to study oscillation in \cite{BB2}?

\item
Prove or disprove the following stability test.

If 
$$|c(t)|\leq c_0<1, ~~a(t)\geq a_0>0,~~ |b(t)|\leq a_0,~~ \int_{h_0(t)}^t a(s)ds\leq \frac{1}{e} $$ for $t$ large enough 
then the equation
$$
\dot{x}(t)=-a(t)x(h_0(t))+b(t)x(h(t))+c(t)\dot{x}(g(t))
$$
is exponentially stable.

By Proposition \ref{proposition5}, this result is true for $h_0(t)\equiv t, h(t)=t-\tau, g(t)=t-\sigma$.

\item
Prove or disprove the following statement.

If $|a(t)|\leq a_0<1, b(t)\geq b_0>0$ and equation (\ref{1}) with bounded delays  has a non-oscillatory solution
then this equation is exponentially stable.

\item
Find exponential stability conditions for equation (\ref{1}) which for the case $g(t)\equiv t$ are 
reduced to the condition $$\displaystyle \int_{h(t)}^t \frac{b(s)}{1-a(s)}ds<\frac{3}{2}$$ (in the case of continuous parameters)
and to the condition $$\displaystyle \int_{h(t)}^t \frac{b(s)}{1-a(s)}ds<1+\frac{1}{e}$$ (in the case of measurable parameters).

\item
Find exponential stability conditions for equation (\ref{1}) which for the case $a(t)\equiv 0$ are 
reduced to the condition $\displaystyle \int_{h(t)}^t b(s)ds<\frac{3}{2}$ (continuous parameters)
and to the condition $\displaystyle \int_{h(t)}^t b(s)ds<1+\frac{1}{e}$ (measurable parameters).

\item
Consider equation (\ref{17}), where $0< b_0\leq b_k(t)\leq b_k$, $t-h_k(t)\leq \tau_k$, $k=1,\dots,m$.
Find exponential stability conditions for the equation  which in the case $a(t)\equiv 0$ 
reduce to the condition $\displaystyle \sum_{k=1}^m b_k\tau_k\leq 1$, and in the case $b_k(t)\equiv b_k$
to the condition $\displaystyle \sum_{k=1}^m b_k\tau_k<\frac{3}{2}$ (continuous  parameters).
If equation (\ref{17}) has measurable parameters then  $\frac{3}{2}$ is replaced by $1+\frac{1}{e}$.

\item
Consider the logistic type neutral equation
\begin{equation}\label{33}
\dot{x}(t)-a(t)\dot{x}(g(t))=-b(t)x(h(t))(1+x(t)),
\end{equation}
where $|a(t)|\leq a_0<1, b(t)\geq b_0>0$.  
Find global exponential stability conditions for equation (\ref{33}) which in the case $g(t)\equiv t$ become  
$$\displaystyle \int_{h(t)}^t \frac{b(s)}{1-a(s)}ds<\frac{3}{2}$$ for continuous parameters
and  $$\displaystyle \int_{h(t)}^t \frac{b(s)}{1-a(s)}ds<1+\frac{1}{e}$$ for measurable parameters.
%
Similarly, establish global exponential stability conditions for equation (\ref{33}) which in the case $a(t)\equiv 0$ 
reduce to  $$\displaystyle \int_{h(t)}^t b(s)ds<\frac{3}{2}$$ for continuous parameters
and to $$\displaystyle \int_{h(t)}^t b(s)ds<1+\frac{1}{e}$$ for measurable parameters.
In general, can a theory similar to \cite{BB4} be developed in the neutral case?

\end{enumerate}

\section{Acknowledgments}

The authors are very grateful to the anonymous referee whose thoughtful comments significantly contributed 
to the paper.

{\small
{\em Authors' addresses}:

{\em Leonid Berezansky}, Department of Mathematics,
Ben-Gurion University of the Negev, 
Beer-Sheva 84105, Israel
 e-mail: \texttt{brznsky@\allowbreak math.bgu.ac.il}.

{\em Elena Braverman}, Department of Mathematics and Statistics,
University of Calgary, 
2500 University Drive N.W., Calgary, AB T2N 1N4, Canada
 e-mail: \texttt{maelena@\allowbreak ucalgary.ca}.
}

\end{document}